\documentclass[%
 aip,
 amsmath,amssymb,
 reprint,%
]{revtex4-1}

\usepackage{amsthm}
\usepackage{enumerate}

\usepackage{tikz}
\usetikzlibrary{positioning}

\usepackage{graphicx}
\usepackage{dcolumn}
\usepackage{bm}

\usepackage[utf8]{inputenc}
\usepackage[T1]{fontenc}
\usepackage{mathptmx}

\newtheorem{prop}{Proposition}

\newtheorem*{thm*}{Theorem}

\newtheorem{definition}{Definition}

\begin{document}

\preprint{AIP/123-QED}

\title[Rethinking the Definition of Rate-Induced Tipping]{Rethinking the Definition of Rate-Induced Tipping}

\author{Alanna Hoyer-Leitzel}
\email{ahoyerle@mtholyoke.edu}
 \affiliation{Department of Mathematics and Statistics, Mount Holyoke College, South Hadley, MA 01075 USA}
\author{Alice N Nadeau}
\email{a.nadeau@cornell.edu}
\affiliation{Department of Mathematics, Cornell University, Ithaca, NY 14853, USA}

\date{\today}

\begin{abstract}
The current definition of rate-induced tipping is tied to the idea of a pullback attractor limiting  in forward and backward time to a stable quasi-static equilibrium.  Here we propose a new definition that encompasses the standard definition in the literature for certain scalar systems and includes previously excluded $N$-dimensional systems that exhibit rate-dependent critical transitions.
\end{abstract}

\maketitle

\begin{quotation}
Tipping points and critical transitions, characterized by qualitative changes in a system due to small changes in its parameters,\cite{lenton2008,scheffer2009} have  been topics of interest to  both the scientific community and the public, especially in the last two decades.  Reasons for this interest perhaps stems from the wide variety of chemical, biological, physical, and societal systems which exhibit such behavior.  Examples include organ failure, desertification, and  runaway ice cover.\cite{An2012,Chen2015,Lenton2013}
In mathematical models, tipping can be caused by different mechanisms, including the rate at which parameters are changing, called \emph{rate-induced tipping}.\cite{Ashwin2012}  
Here we discuss several examples of systems exhibiting rate-induced critical transitions but which fall outside of the standard set of assumptions that encompass the theoretical developments of the rate-induced tipping field. 
As such, one cannot technically call the transitions we observe in these systems ``rate-induced tipping.'' 
In light of these examples, we propose a new definition of rate-induced tipping as the loss of forward asymptotic stability of a nonautonomous attractor.  Encouragingly, this new definition is equivalent to the current definition for asymptotically constant scalar systems\cite{Kuehn2020} although it is not known if it generalizes to asymptotically constant $N$-dimensional systems or $N$-dimensional systems with other types of parameter change.
\end{quotation}

\section{Introduction}

Tipping points in the scientific literature are characterized by a sudden, qualitative shift in the behavior or state of the system due to a relatively small change in inputs.\cite{lenton2008,scheffer2009}  
In natural systems this behavior can have a variety of underlying causes, such as changing external conditions or self-excitation.  For example, nutrient run-off from agricultural practices can cause downstream lakes to ``tip'' from a healthy state to a eutrophic state resulting in excessive algae growth and death of animal life.  In mathematical models, tipping may be caused by a bifurcation due to a change in parameter values, noise in the system, or the rate at which parameters are changing.\cite{Ashwin2012}  This last cause, called rate-induced tipping (sometimes \emph{rate-dependent tipping}, \emph{rate tipping},  or simply \emph{R-tipping}), can tip a system to a different state even when the parameter change does not pass through a bifurcation point of the system.

The mathematical study of rate-induced tipping has almost exclusively been concerned with deterministic, continuous time systems of the form 
\begin{align}
\dot x=f(x,\lambda(rt),t),\quad x\in\mathbb R^N, \ \lambda(rt)\in\mathbb R^M
\label{EQ-system}
\end{align}
especially where $N=1$ and $M=1$. Although recent research has been conducted on understanding rate-induced tipping in noisy systems\cite{Ritchie2016, Ritchie2017} or time-dependent maps,\cite{Kiers2020} here we will focus only on deterministic, continuous time systems.  Finding rate-induced tipping in system \eqref{EQ-system} amounts to varying the parameter $r$, referred to as the \emph{rate parameter}, and assessing the system's sensitivity to the change of the parameter $\lambda(rt)$ in time. 

It is natural to compare the time-dependent system \eqref{EQ-system} with the corresponding family of autonomous systems $\dot{x}=f(x,\lambda)$ where $\lambda$ is constant. Although recent studies have considered tipping in systems with periodic orbits,\cite{Alkhayoun2018,Hahn2017,Okeefe2020} it is standard to assume that equilibria of the corresponding autonomous system are hyperbolic.  In the hyperbolic context, one may consider the curve of equilibria $x_*(\lambda(rs))$ generated by the stationary system $\dot x=f(x,\lambda(rs))$ as $s$ is varied,  called a \emph{quasi-static equilibrium} or \emph{QSE}.  

The definition of rate-induced tipping has evolved over time as mathematicians try to describe the phenomenon rigorously.  Rate tipping was originally determined by how close a trajectory remained to a stable QSE: trajectories that left a predetermined radius around a stable QSE had tipped.\cite{Ashwin2012}  This seemed like a natural definition as several early examples were found where the boundary of the basin of attraction stayed a constant distance away from the stable QSE. However, in other systems, solutions may drift away from the stable QSE as the rate parameter $r$ is varied without having a critical transition,\cite{siteur2016ecosystems} leaving the definition lacking.  In Section \ref{Section-drift}, we will discuss a system where solutions exhibit this drifting behavior relative to the stable QSE.

The current working definition of rate tipping also relies on behavior of trajectories of the system in relation to the system's stable QSEs.\cite{Ashwin2017,Kiers2019} In systems where the change in the parameter $\lambda(rt)$ is asymptotically constant---limiting to $\lambda_{\pm}$when $t\to\pm\infty$---the definition of rate-induced tipping relies on the limiting behavior of (local) pullback attractors of the system.  A curve $\gamma_P(t)$ is (local) pullback attracting if trajectories $x(t;x_0,t_0)$ get closer to $\gamma_P(t)$ as the initial condition $x_0$ is started further back in time, i.e. for $x_0$ in some ball around $\gamma_P(t)$,
\begin{align}
\lim_{s\to-\infty}|x(t;x_0,s)-\gamma_P(t)|=0 \, \, \text{ for all } t.
\label{EQ-pullback-attract}
\end{align}
Ashwin, Perryman, and Wieczorek show that for each stable hyperbolic equilibrium of the autonomous system $\dot x=f(x,\lambda_-)$, there is a corresponding solution $x_P(t)$ to \eqref{EQ-system} which is local pullback attracting and converges to the stable QSE limit $x_*(\lambda_-)$ as $t\to-\infty$.\cite{Ashwin2017}  If the local pullback attractor converges to the stable QSE limit $x_*(\lambda_+)$ as $t\to\infty$, the solution is said to \emph{end-point track} the QSE. 
Rate-induced tipping is said to occur if the local pullback attractor $x_P(t)$ does not end-point track the stable QSE, i.e. does not limit to the forward limit of the corresponding equilibrium $x_*(\lambda_+)$ as $t\to\infty$ as the rate parameter $r$ is varied.  

One helpful characterization of nontipping for asymptotically constant scalar systems is \emph{forward basin stability}, when the minimal closed set in state space  containing the values of the curve traced out by the stable QSE up to time $s$ is in the basin of attraction of $x_*(\lambda(s))$ for all $s\in\mathbb R$.  Pullback attractors of scalar systems which are forward basin stable will always end-point track their respective stable QSE, and thus do not experience rate-induced tipping. The $N$-dimensional ($N>1$) generalization of forward basin stability is \emph{forward inflowing stable}.\cite{Kiers2019} A detailed discussion of forward inflowing systems follows in Section~\ref{Section-n-dim}. However, much like end-point tracking, the ideas of foward basin stability or foward inflowing stability are (currently) defined only for asymptotically constant parameter changes. 

Here we consider systems with locally bounded parameter change, where $\lambda(rt)$ is bounded on any finite time interval and smooth enough to ensure existence and uniqueness of solutions. This includes the type of parameter change seen in tipping studies where $\lambda(rt)$ is asymptotically constant and also systems where $\lambda(rt)$ may be unbounded when considered on the whole real line, such as $\lambda(rt)=rt$. In the context of locally bounded parameter change, we say a solution $x(t;x_0,t_0)$ \emph{end-point tracks} a quasistatic equilibrium $Q(t)$ if
\[\lim_{s\to\infty}|x(s;x_0,t_0)-Q(s)|=0.\]
This is a slight generalization of the current definition of end-point tracking 
 to account for solutions and QSEs which may have infinite limits; however, this definition is equivalent to the current one when the limit of the QSE is finite.

In the context of locally bounded parameter change, we find that the definition of rate-induced tipping as a pullback attractor not end-point tracking a stable QSE is too restrictive.  In Section \ref{Section-unbounded}, we give several examples demonstrating that the long term behavior of solutions relative to a the (local) stable QSE does not impact whether the system has a critical transition caused by varying the rate parameter $r$, including an example that is forward basin stable but does not endpoint-track a stable QSE. Because of this behavior, it may be more appropriate to define rate-induced tipping in terms of the loss of forward attraction of the (local) pullback attractor.  Recent advances in the literature have shown that this loss of forward attraction is equivalent to end-point tracking for scalar equations with asymptotically constant parameter change\cite{Kuehn2020} and a specific family of systems where $f(x,\lambda(rt),t)$ is bounded for all $t\in\mathbb R$.\cite{Longo2021} Further, loss of forward attraction encompasses rate-induced tipping-like behavior in $N$-dimensional systems where the current definition is not applicable, as we will see in Section \ref{Section-unbounded}. 

In nonautonomous systems, there are two types of attraction: pullback attraction and forward attraction.  In contrast with pullback attraction, a curve $\gamma_F(s)$ is  (local) forward attracting if trajectories $x(t;x_0,t_0)$ get closer to $\gamma_F(t)$ as they flow forward in time, i.e. for $x_0$ in some ball around $\gamma_F(s)$,
\begin{align}
\lim_{s\to\infty}|x(s;x_0,t_0)-\gamma_F(s)|=0 \, \, \text{ for all } t_0.
\label{EQ-forward-attract}
\end{align}
In autonomous systems, the pullback and forward definitions of attraction are equivalent; however, in nonautonomous systems pullback attractors need not be forward attracting and forward attractors need not be pullback attracting. \cite{Kloeden2011,Langa2006}  When an attractor is both forward and pullback attracting, it is called a \emph{uniform attractor}.\cite{Kloeden2011}

We propose a new definition of rate tipping as follows
\begin{definition}
\label{def-tipping}
System \eqref{EQ-system} undergoes \emph{rate-induced tipping} at $r=r^*$ when a (local) forward attracting pullback attractor $\gamma(t)$ of the system goes through a nonautonomous bifurcation at $r=r^*$ causing the pullback attractor to lose its forward attraction.
\end{definition} 
\noindent Note that for the bifurcation at $r=r^*$ in the proposed definition, the pullback attractor may remain in the system and only lose its forward stability \cite{Kuehn2020,Longo2021} or the pullback attractor may be annihilated in a collision with other hyperbolic globally defined solutions (solutions which exist for all time and whose variational equations exhibit an exponential dichotomy on $\mathbb R$).\cite{Coppel2006}  We see examples of  annihilation of the pullback attractor in Sections \ref{Section-SN}, \ref{Section-SN-local}, and \ref{Section-unbounded-n}.  
If there are multiple forward attracting pullback attractors in the system, varying the rate parameter $r$ to the critical value $r^*$ needn't cause all pullback attractors to lose  their forward attraction.  In Section \ref{Section-SN-local}, we show an example of the localness of the induced bifurcation.

The remainder of this note is laid out as follows.  In Section~\ref{Section-asymptotic} we discuss asymptotically bounded systems.  In Sections~\ref{Section-KL} and \ref{Section-Longo-etal} we briefly summarize recent papers showing that in some scalar systems the definition for rate-induced tipping proposed in Definition \ref{def-tipping} above is equivalent to the current definition of rate-induced tipping used in the literature.\cite{Kuehn2020} In Section~\ref{Section-n-dim}, we prove that the conditions that guarantee that the asymptotically constant $N$-dimensional system end-point tracks for all $r>0$ also guarantee that the pullback attractor is forward attracting for all $r>0$. 
In Section~\ref{Section-unbounded} we discuss systems with locally bounded parameter change.    
In Sections~\ref{Section-drift}, we present a scalar system with a pullback attractor that does not end-point track the stable but which also does not undergo a critical transition as the rate parameter is varied. In Sections~\ref{Section-SN}, \ref{Section-SN-local} and \ref{Section-unbounded-n}, we provide examples of systems which undergo rate-dependent critical transitions due to a pullback attractor losing its forward stability as the rate parameter is varied but which would not be called ``rate-induced tipping'' under the current definition because the pullback attractors  do not ``endpoint track" the stable QSEs for any value of the rate parameters. In Section~\ref{Section-proof}, we prove a more general statement encompassing the examples from Sections~\ref{Section-SN}, \ref{Section-SN-local}, and \ref{Section-unbounded-n}.

\section{Examples when $f(x,\lambda(rt),t)$ is Bounded for $t\in \mathbb R$}
\label{Section-asymptotic}

\subsection{Asymptotically Constant Scalar Equations}
\label{Section-KL}

Among the results of a recent paper by Kuehn and Longo is a theorem showing that for scalar systems with asymptotically constant parameter change, and under other standard assumptions used in the tipping literature, a pullback attractor does not endpoint track its associated stable QSE when it loses its forward asymptotic stability as the rate parameter is varied.\cite{Kuehn2020} The authors show that this occurs when the forward attracting pullback attractor collides with a pullback repeller at a critical value of $r=r^*$ defining the tipping point.

More precisely, the study assumes that $f$ in \eqref{EQ-system} has twice differentiable partial derivatives that are continuous and $x$ and $\lambda$ are both scalars
, i.e. $f\in\mathcal C^2(\mathbb R\times\mathbb R,\mathbb R)$. Additional assumptions include that the autonomous problem $\dot x=f(x,\lambda)$ has a stable hyperbolic equilibrium which depends continuously on $\lambda$ for all $\lambda\in[\lambda_-,\lambda_+]$,  that all other equilibria of the autonomous problem are hyperbolic for all $\lambda\in[\lambda_-,\lambda_+]$ except for at most finitely many points $\lambda\in(\lambda_-,\lambda_+)$, and that $\lambda(rt)$ is strictly increasing and contained in the interval $[\lambda_-,\lambda_+]$. 
Let
\begin{align*}
r^*:=\sup\{r>0 \ | \ &x^{\rho}_-(\cdot)\text{ is uniformly} \\
&\text{asymptotically stable for all }0<\rho\leq r\},
\end{align*}
 where $x^{\rho}_-(\cdot)$ is a pullback attractor limiting to a stable QSE in backward time. The authors show that then $x^{r}_-$ end-point tracks the respective curve of quasi-static equilibria for $0<r<r^*$ and $r^*<\infty$ if and only if $x^{r^*}_-$ is globally defined but not uniformly asymptotically stable and it does not end-point track the QSE (Kuehn and Longo, Theorem 3.6\cite{Kuehn2020}).  At this critical value of the rate parameter $r$, the forward attracting pullback attractor in the system collides with a forward repelling pullback repeller.

This result by Kuehn and Longo demonstrates that Definition 1 is equivalent to the current definition of rate-induced tipping for scalar systems under these assumptions. 

\subsection{Bounded Asymptotically Nonautonomous Scalar Equations}

\label{Section-Longo-etal}

Longo et al describe rate-induced tipping in a different class of systems
\[\dot x=-\left(x-\frac{2}{\pi}\text{arctan}(rt)\right)^2+p(t)\]
where $p(t)$ is bounded.\cite{Longo2021} Here they investigate behavior of the system due to different classes of functions $p(t)$.  In their conclusion, the authors note that extending their results by replacing $\text{arctan}(rt)$ with some other odd function $\lambda(rt)$ which is asymptotically constant is far from straightforward and remains an open question.

Among the results, the authors find that for the class of functions $p(t)$ where
\[\dot x=-x^2+p(t)\]
has exactly two hyperbolic solutions, there is a critical rate $r^*$ for which a local pullback attractor and a local pullback repeller come together and form a special globally defined, bounded solution which is attracting in the past and repelling in the future. When $r<r^*$, there exists a bounded, globally defined local pullback attractor which limits forward and backward in time to the attractors of the limiting systems
\[\dot x=-(x+1)^2+p(t)\]
and
\[\dot x=-(x-1)^2+p(t),\]
respectively.  Additionally, there is a bounded, globally defined, local pullback repeller of the full scalar equation which limits forward and backward in time to the repellers of the limiting systems.  Taking $r>r^*$ results in the loss of all globally defined solutions. Intriguingly, Longo et al show that the critical rate $r^*$ is not unique and that there can be several intervals in $r$-space with existence of the attractor/repeller pair or nonexistence of all bounded solutions. Finally, the authors show that rate-tipping occurs in this system if and only if there is a nonautonomous saddle node bifurcation, which is found in a linked system through a change of variables.  

\subsection{Asymptotically Constant $N$-Dimensional Systems, $N>1$}
\label{Section-n-dim}

It is not yet known if the results showing the equivalence of a nonautonomous bifurcation and end-point tracking for scalar equations generalize to $N$-dimensions, $N>1$.  However, below we show that the same conditions that guarantee end-point tracking in higher dimensional systems when $\lambda(rt)$ is asymptotically constant also guarantee that the pullback attractor is forward attracting.

Kiers and Jones give a condition to guarantee that the pullback attractor of a system in $N$-dimensions ($N>1$) with asymptotically constant parameter change end-point tracks the stable QSE.\cite{Kiers2019}  Here we show that the assumptions establishing this result are strong enough to also guarantee that the pullback attractor is a forward attractor for all positive values of the rate parameter.  This shows that the same conditions that guarantee there is no rate-induced tipping are the same conditions that guarantee the pullback attractor is also a forward attractor.  In order to prove this result, we need to recall two definitions from earlier studies.

\begin{definition}[Stable path\cite{Ashwin2017}]
Given an asymptotically constant parameter shift $\Lambda(s)$ satisfying $\lambda_-<\lambda(s)<\lambda_+$, $\lim_{s\to\pm\infty}=\lambda_\pm$, and $\lim_{s\to\pm\infty}d\lambda(s)/ds=0$, if for all fixed $s\in\mathbb R$, $X(s)$ is an attracting equilibrium of the autonomous system
\[\dot x=f(x,\lambda(s)),\]
then we say that $(X(s),\lambda(s))$ is a \emph{stable path} of the augmented autonomous system
\begin{equation}
\begin{aligned}
\dot x&=f(x,\lambda(s)),\\
\dot s&=r.
\label{EQ-augmented}
\end{aligned}
\end{equation}
Furthermore we define
\[X_\pm=\lim_{s\to\pm\infty}X(s).\]
\end{definition}

\begin{definition}[Forward inflowing stable\cite{Kiers2019}]
\label{def-fis}
The stable path $(X(s),\lambda(s))$ from $X_-$ to $X_+$ is \emph{forward inflowing stable} (FIS) if for each $s\in\mathbb R$ there is a compact set $K(s)$ such that
\begin{enumerate}
\item $X(s)\in\text{Int}K(s)$ for all $s\in\mathbb R$;
\item if $s_1<s_2$ then $K(s_1)\subset K(s_2)$;
\item if $x\in\partial K(x)$ then there exists $t_0>0$ so that $x(t:x,s)\in\text{Int}K(s)$ for all $t\in (s,s+t_0)$
\item $X_\pm\in \text{Int}K_\pm$ where $K_-=\bigcap_{s\in\mathbb R}K(s)$ and $K_+=\overline{\bigcup_{s\in\mathbb R}K(s)}$; and
\item $K_+\subset\mathbb B(X_+,\lambda_+)$ is compact.
\end{enumerate}
\end{definition}

\begin{prop}
If the stable path $(X(s),\lambda(s))$ from $X_-$ to $X_+$ is forward inflowing stable, then the pullback attractor to $X_-$ is end-point tracking and forward attracting for all $r>0$.
\end{prop}

\begin{proof}
The result regarding end-point tracking is proved by Kiers and Jones\cite{Kiers2019} (see Proposition 1 in that study). The following proof showing that the pullback attractor is forward attracting uses much of the machinery developed in that proof and other results in the same paper.

Fix $r>0$ and suppose the stable path $(X(s),\lambda(s))$ from $X_-$ to $X_+$ is forward inflowing stable.  Then there exists sets $K(s)$ satisfying Definition \ref{def-fis}. Let $K=\bigcup_{s\in\mathbb R}K(s)\times\{s\}$ and note that $K$ is forward invariant under the flow of the autonomous augmented system \eqref{EQ-augmented} as shown by Kiers and Jones.\cite{Kiers2019}

Under these assumptions there exists a pullback attractor $\gamma_{r}(t)$ limiting to $X_-$ as $t\to-\infty$ (see Theorem 2.2\cite{Ashwin2017}). By the same arguments in the proof of Proposition 1,\cite{Kiers2019} $\gamma_{r}(t)$ is in $K(rt)$ and $K_+$ for all $t\in\mathbb R$ and converges to $X_+$ as $t\to\infty$.

Given $x_0\in \partial K(rt_0)$, let $x(t;x_0,rt_0)$ be the solution to \eqref{EQ-system} through $x_0$ at time $rt_0$.  By the third property of forward inflowing stable, we know that there is a $s_0>0$ so that $x(t;x_0,rt_0)\in\text{Int} K(rt_0)$ for all $t\in(rt_0,rt_0+s_0)$. Because $K$ is forward invariant, this implies that $x(t;x_0,rt_0)\in\text{Int} K(t)$ and, thus, $K_+$ for all $t>rt_0$.  By Lemma 5\cite{Kiers2019} this implies that $x(t;x_0,rt_0)$ converges to $X_+$ as $t\to\infty$.

Then we see that 
\begin{align*}
\lim_{t\to\infty}&\|x(t;x_0,rt_0)-\gamma_{r}(t)\|\\
&\quad\leq \lim_{t\to\infty}\left(\|x(t;x_0,rt_0)-X_+\|+\|X_+-\gamma_{r}(t)\|\right)=0.
\end{align*}
Since $x_0$ and $t_0$ were arbitrary, we can conclude that $\gamma_{r}(t)$ is forward attracting.
\end{proof}

Note that this result does not rule out that there could be an end-point tracking pullback attractor which is not forward attracting nor does it rule out that there could be a forward attracting pullback attractor which is not end-point tracking. More work is needed to determine if the ideas of not end-point tracking and loss of forward stability are equivalent in $N-$dimensional systems with asymptotically constant parameter change.


\begin{figure}
\begin{center}
\begin{tikzpicture}
  \node (img1)  {\includegraphics[width=0.4\textwidth, trim=0.4cm 0.85cm 0.75cm 0.7cm, clip]{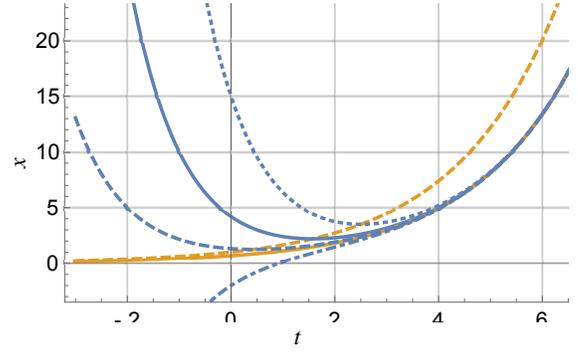}};
  \node[below=of img1, node distance=0cm, yshift=1.1cm] {$t$};
  \node[left=of img1, node distance=0cm, rotate=90, anchor=center, yshift=-1cm] {$x$};
\end{tikzpicture}
      \caption{Time series for $\dot{x}=-(x-e^{rt})$ with $r=1/2$ showing the system's forward attracting pullback attractor (solid yellow), stable QSE (dashed yellow), and selected solutions (blue curves; dotted, solid, dashed, dot-dashed). For $r>0$ and any $\epsilon>0$, the pullback attractor leaves the $\epsilon$-ball around the QSE forward in time.  }
    \label{fig:ex_drift}
    \end{center}
    \end{figure}

\section{Examples when $f(x,\lambda(rt),t)$ is Unbounded for $t\in \mathbb R$}
\label{Section-unbounded}

The theoretical framework for rate tipping is underdeveloped for locally bounded parameter change, where $\lambda(rt)$ is bounded for any finite time interval but may be unbounded on $\mathbb R$.  Below we detail some examples where solutions of systems do not end-point track the stable QSE of the system. In these examples we will see that not tracking the stable QSE is possible whether or not there is a critical transition caused by the rate parameter.  

In the first subsection, we given an example of a system where for any $\epsilon$, the pullback attractor always leaves any $\epsilon$-ball of the stable QSE forward in time but there is no rate-induced critical transition. In the second and third section we give systems where the local pullback attractors do not limit to the stable QSE and there is a rate-dependent critical transition. However, care must be taken because the former does not imply the latter and there are values of the rate parameter for which the pullback attractors do not end-point track the stable QSE but also do not have a critical transition. 

These examples show that rate-dependent transitions in systems with locally bounded parameter change can be independent of the limiting behavior of the pullback attractor in relation to the QSEs.  This independence of rate-dependent critical transitions from end-point tracking is the main reason we are advocating for removing the idea of the QSE from the definition of rate-induced tipping.

\subsection{Scalar Equations }

\subsubsection{Drift away from a stable QSE without a critical transition}
\label{Section-drift}

\noindent Consider system \eqref{EQ-system} with $N=1$ and
\begin{align}
    f(x,\lambda(rt))=-(x-\lambda(rt))
\label{EQ-drift}
\end{align}
with $\lambda(rt)=e^{rt}$ and $r>0$.  Using integrating factors, we see that the family of solutions is
\begin{align*}
x(t;t_0,x_0)=\frac{e^{rt}}{1+r}+C(x_0,t_0)e^{-t}.
\end{align*}
where $C(x_0,t_0)=\left(x_0-\frac{e^{rt_0}}{1+r}\right)e^{t_0}$ depends on the initial condition. Trajectories, the pullback attractor and the QSE for this system are plotted in Figure \ref{fig:ex_drift} for three values of the rate parameter.

For every choice of $r\in\mathbb R$ the solution
\[\gamma_r(t)=\frac{e^{rt}}{1+r}\]
is a global forward attracting pullback attractor. The forward attraction is clear to see from the form of the solutions.  To see that $\gamma_r(t)$ is a pullback attractor for the system, notice that
\begin{align*}
\lim_{s\to-\infty}&|x(t;x_0,s)-\gamma_r(t)|\\
&=\lim_{s\to-\infty}\left|\left(x_0-\frac{e^{rs}}{1+r}\right)e^{s-t}\right|=0.
\end{align*}
In fact, $\gamma_r(t)$ is the unique pullback attractor of the system and is the solid yellow line in Figure~\ref{fig:ex_drift}.

Ashwin et al have shown that scalar, asymptotically constant systems that have pullback attractors that are forward basin stable must end-point track their corresponding stable QSEs.\cite{Ashwin2017}  We see that
\begin{definition}[Forward Basin Stable\cite{Ashwin2017}] A stable path $(X(s),\lambda(s))$ is \emph{forward basin stable} if
\[\overline{\{X(u):u<s\}}\subset \mathbb B(X(s),\lambda(s))\qquad\text{for all }s\in\mathbb R\]
where $\mathbb B(X(s),\lambda(s))$ is the basin of attraction of the stable equilibrium $X(s)$.
\end{definition}
\noindent Note that the system contains a global forward attracting pullback attractor and the stationary system (corresponding autonomous system created by fixing $t$) has a global stable equilibrium for each $t$. In systems where $\lambda$ is asymptotically constant, a global forward attracting pullback attractor is forward basin stable. Thus any generalization of the definition of forward basin stability to systems with locally bounded parameter changes should include the system \eqref{EQ-drift}. 

However, the pullback attractor $\gamma_r(t)$ does not endpoint-track the QSE in this example.  Indeed, this equation has one quasi-static equilibrium given by $Q(t)=\lambda(rt)=e^{rt}$ which is stable. 
For the distance between the pullback attractor and the QSE is given by
\[|\gamma_r(t)-Q(t)|=\left|\frac{re^{rt}}{1+r}\right|\]
which grows without bound as time advances provided $r>0$ as assumed. For any $\epsilon>0$ the pullback attractor eventually leave the $\epsilon$-radius of the stable QSE. Thus, the pullback attractor does not end-point track the QSE and we see that forward basin stable does not imply end-point tracking in systems with non-asymptotically constant parameter changes.

\textbf{Remark.}  Because the pullback attractor does not end-point track the stable QSE forward in time when $r>0$, the current definition of rate-induced tipping says that the system tips for all $r>0$. Even worse, all solutions would be said to tip even under the weaker definition of \emph{tracking}\cite{Ashwin2012} where a solution need only be in some $\epsilon$-ball of the stable QSE in the limit as $t\to\infty$. However, for all $r>0$, there is a unique, global forward attracting pullback attractor and all solutions are defined for all time.  There are no critical transitions as a result of varying the rate parameter $r$ in this system and so it would be unfortunate to say there is rate-induced tipping in this scalar equation.

\begin{figure*}
\begin{center}
\begin{tikzpicture}
  \node (img1)  {\includegraphics[width=0.3\textwidth, trim=1.8cm .8cm 6.2cm .8cm, clip]{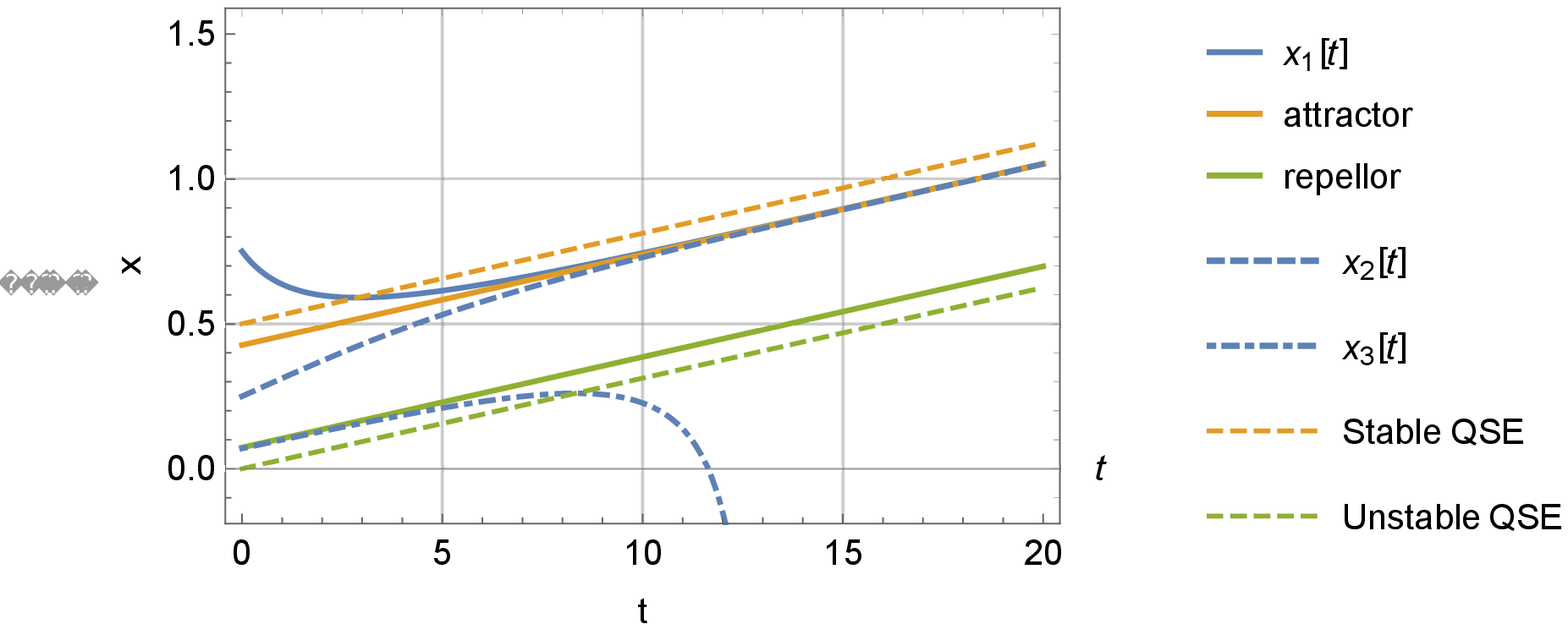}};
  \node[above=of img1, node distance=0cm, xshift=-2cm, yshift=-1.7cm] {\bf A};
  \node[below=of img1, node distance=0cm, yshift=1.1cm] {$t$};
  \node[left=of img1, node distance=0cm, rotate=90, anchor=center, yshift=-1cm] {$x$};
  \node[right=of img1,yshift=0.05cm, xshift=-.5cm] (img2)  {\includegraphics[width=0.31\textwidth, trim=0.4cm 0.85cm 1.0cm 0.7cm, clip]{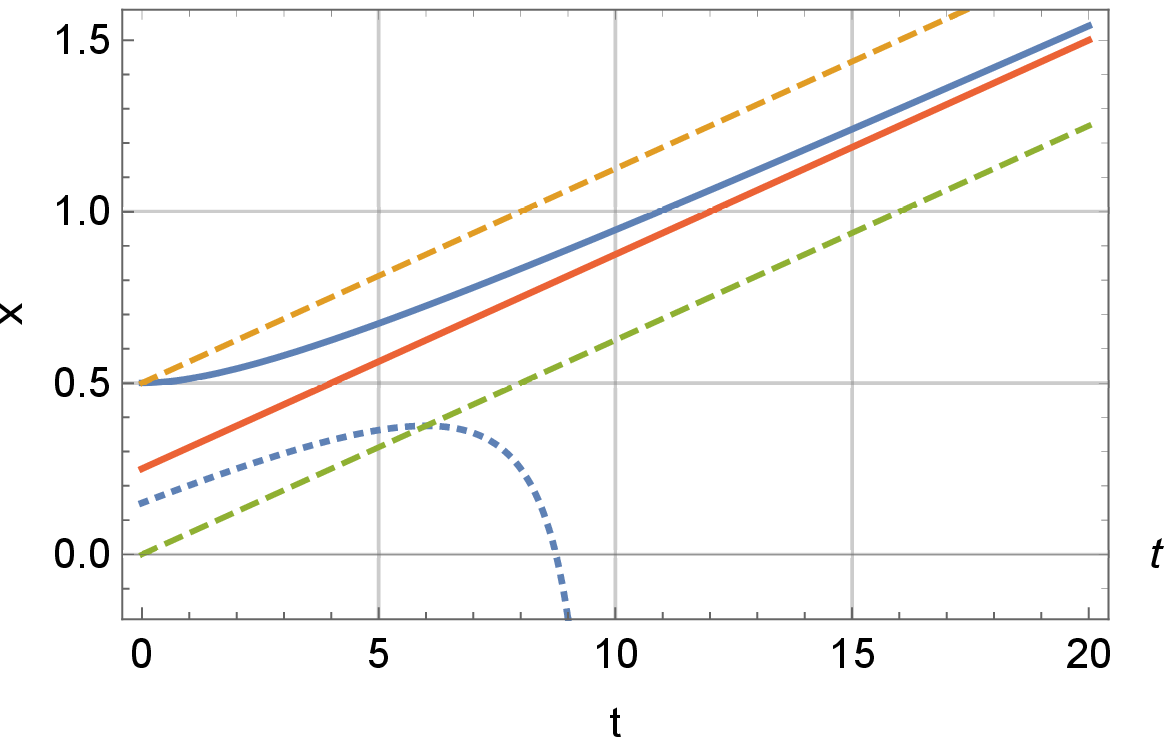}};
  \node[above=of img2, node distance=0cm, xshift=-2cm, yshift=-1.7cm] {\bf B};
  \node[below=of img2, node distance=0cm, yshift=1.1cm] {$t$};
  \node[left=of img2, node distance=0cm, rotate=90, anchor=center, yshift=-1cm] {$x$};
   \node[right=of img2,yshift=0cm,  xshift=-.5cm] (img3)  {\includegraphics[width=0.3\textwidth,  trim=1.8cm .8cm 5.8cm .8cm, clip]{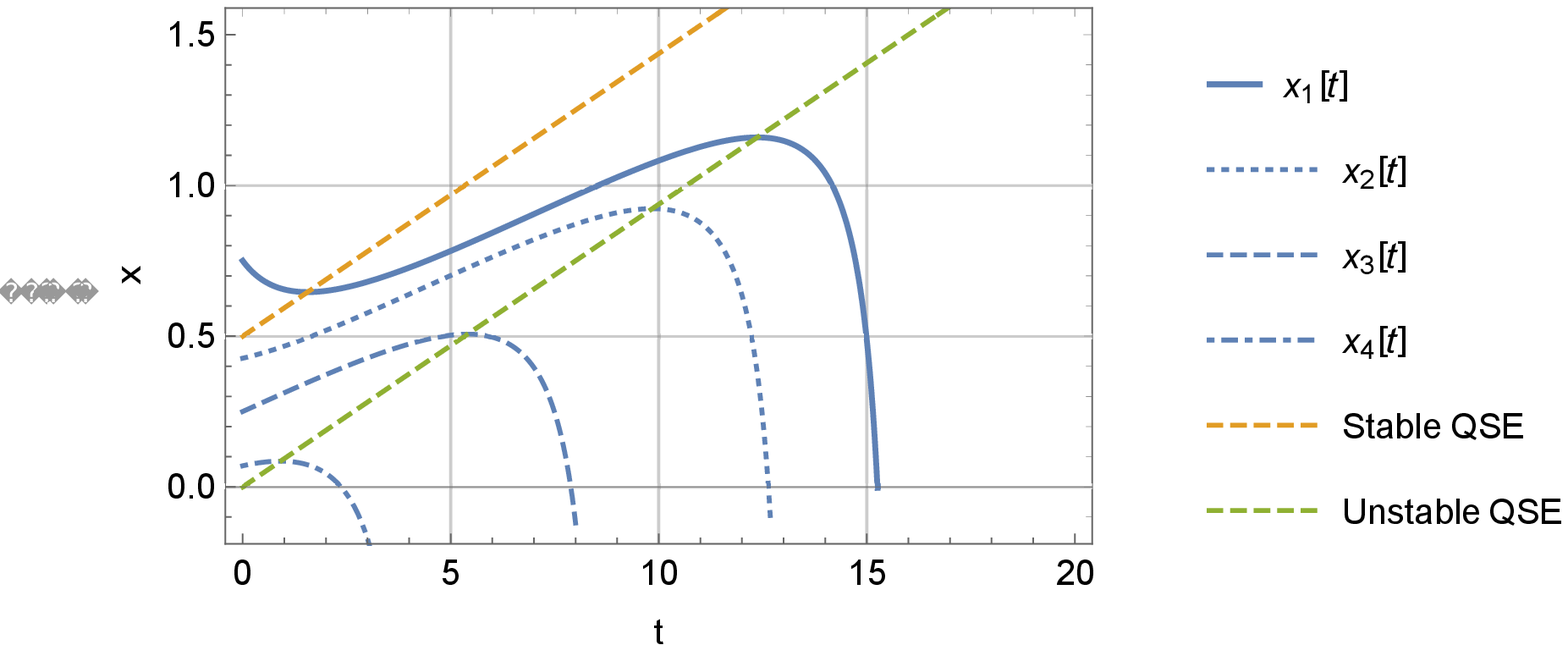}};
  \node[above=of img3, node distance=0cm, xshift=-2cm, yshift=-1.7cm] {\bf C};
  \node[below=of img3, node distance=0cm, yshift=1.1cm] {$t$};
  \node[left=of img3, node distance=0cm, rotate=90, anchor=center, yshift=-1cm] {$x$};
\end{tikzpicture}
       \caption{Time series for $\dot{x}=-(x-rt)(x-rt-1/2)$ showing the system's forward attracting pullback attractor (solid yellow), stable QSE (dashed yellow), forward repelling pullback repeller (solid green), unstable QSE (dashed green), and selected solutions (blue curves; dotted, solid, dashed, dot-dashed). \textbf{A:} $r=1/32$, \textbf{B:} $r=r^*=1/16$, \textbf{C:} $r=3/32$.  As $r$ increases to the critical value $r^*=1/16$ in \textbf{B}, the forward attracting pullback attractor and forward repelling pullback repeller collide, and at the critical value $r^*$, they are replaced by a hyperbolic globally defined solution that is neither attracting nor repelling (red).  }
    \label{fig:ex_SN}
    \end{center}
    \end{figure*}

\subsubsection{Drift away from a Stable QSE with a critical transition}
\label{Section-SN}

\noindent Consider system \eqref{EQ-system} with $N=1$ and
\[f(x,\lambda(rt))=-(x-\lambda(rt))(x-\lambda(rt)-\mu)\]
with $\lambda(rt)=rt$, pictured for three values of $r$ in Figure \ref{fig:ex_SN} with $\mu=1/2$.  A similar form of this equation has been analyzed in the literature (see example 3a in Ashwin et al, (2012)\cite{Ashwin2012}).  Solutions to the equation are straightforward to find once the equation is converted into co-moving coordinates.  Namely, let $y(t)=x(t)-\lambda(rt)-\mu/2$, then the co-moving equation is the autonomous equation
\begin{align*}
\dot y=-(y-\mu/2)(y+\mu/2)-r=-y^2+\left(\frac{\mu^2}{4}-r\right)
\end{align*}
with family of solutions for $r\leq\mu^2/4$ given by
\begin{align*}
y(t)=\begin{cases}
\rho\text{tanh}\left(\rho(t+C_1)\right), & |y_0|<\rho\\
\rho\text{coth}\left(\rho(t+C_2)\right) ,& |y_0|>\rho
\end{cases}
\end{align*}
where
\begin{equation}
\begin{aligned}
\rho&=\sqrt{\frac{\mu^2}{4}-r},\\ 
C_1(y_0,t_0)&=\frac{1}{2\rho}\log\left(\frac{y_0+\rho}{\rho-y_0}\right)-t_0, \quad\text{and} \\
C_2(y_0,t_0)&=\frac{1}{2\rho}\log\left(\frac{y_0+\rho}{y_0-\rho}\right)-t_0. 
\end{aligned}
\label{constants}
\end{equation}
The constant solutions $y(t)=\pm\rho$ are equilibria of the equation. Notice that if the initial condition satisfies $|y_0|>\rho$, then there is a finite time singularity at $t=-C_2$.  If $y_0>\rho$ this singularity happens in backward time and if $y_0<-\rho$ this singularity happens in forward time.

Transforming back to the stationary coordinates, we see that when $r<\mu^2/4$ solutions take the form
\begin{align*}
x(t)=\begin{cases}
rt+(\mu/2)+\rho\text{tanh}\left(\rho(t+C_1)\right), & |x_0-rt_0-\mu/2|<\rho\\
rt+(\mu/2)+\rho\text{coth}\left(\rho(t+C_2)\right) ,& |x_0-rt_0-\mu/2|>\rho
\end{cases}.
\end{align*}
where $C_1=C_1(x_0-rt_0-\mu/2,t_0)$ and $C_2=C_2(x_0-rt_0-\mu/2,t_0)$ as given in \eqref{constants}. For this range of $r$ there is a forward attracting pullback attractor given by
\begin{align*}
\gamma_r(t)=rt +(\mu/2)+\rho,
\end{align*} 
corresponding to the stable equilibrium in the co-moving system.  For the same values of $r$ there is a forward repelling pullback repeller given by
\begin{align*}
\zeta_{r}(t)=rt +(\mu/2)-\rho
\end{align*} 
corresponding to the unstable equilibrium in the co-moving system. Notice that only solutions with initial conditions between the pullback attractor and pullback repeller ($|x_0-rt_0-\mu/2|<\rho$) are defined for all time. Solutions with initial conditions outside of this set have finite time singularities forward or backward in time.

The stable and unstable QSEs of the equation are given by 
\[Q_S(t)=rt+\mu,\quad Q_U(t)=rt,\]
respectively. The distance between the pullback attractor and the stable QSE is given by
\[|\gamma_r(t)-Q_S(t)|=\frac{\mu}{2}-\sqrt{\frac{\mu^2}{4}-r}\]
which is greater than $0$ for $r>0$ and moves away from the stable QSE as $r$ increases. In particular, for any $0<\epsilon<\mu/2-\rho$, all solutions will leave the $\epsilon$-ball around the stable QSE forward in time. The pullback attractor does not end-point track the QSE for any $r>0$.

However, the system has a  nonautonomous saddle node bifurcation where the pullback attractor collides with the pullback repeller when $r=r^*=\mu^2/4$. For this value of $r$, the pullback attractor $\gamma_r$ loses forward stability (see Figure~\ref{fig:ex_SN}b~and~\ref{fig:ex_SN}c), which can be seen in both the stationary coordinates and the co-moving coordinates.  In particular, we see that  for this $r$ the value of $\rho=\sqrt{\mu^2/4-r}$ is zero.  Then solutions with $x_0>rt_0+\mu/2$ converge to the curve $\gamma_r(t)=rt+\mu/2$ in forward time and solutions with $x_0<rt_0+\mu/2$ diverge from the curve $\gamma_r(t)=rt+\mu/2$ in forward time towards negative infinity and have a finite time singularity at $t=-C_2(x_0,t_0)$.

When $r>\mu^2/4$,  there are no equilibria, stable or otherwise, in the autonomous co-moving system nor any globally defined solutions in the nonautonomous system.  In the nonautonomous system,  the family of solutions is given by
\begin{align*}
x(t)=rt+(\mu/2)-\tilde\rho\tan(\tilde\rho(t+C_3))
\end{align*}
where 
\begin{align*}
\tilde\rho&=\sqrt{r-\mu^2/4} \quad \text{and}\\
 C_3(x_0,t_0)&=-\frac{1}{\tilde\rho}\arctan\left(\frac{x_0-rt_0-\mu/2}{\tilde\rho}\right)-t_0.
 \end{align*}
All solutions blow up in finite time so we see that the pullback attractor has been annihilated by the collision with the pullback repeller at $r=r^*=\mu^2/4$.

The system has a critical transition induced by the rate parameter at $r=r^*:=\mu^2/4$ whereby the number of globally defined solutions of the system changes. For $r<r^*$ there are an uncountable number of globally defined solutions: solutions with initial conditions between the pullback attractor $\gamma_r(t)$ and the pullback repeller $\zeta_r(t)$. Solutions which start above the attractor or below the repeller have finite time singularities either backward or forward in time, respectively. At $r=r^*$ there is exactly one globally defined solution: the curve $rt+\mu/2$. This solution is neither attracting nor repelling in either the forward or pullback sense.  It is, however, a hyberbolic solution (its variational equation admits an exponential dichotomy).  When $r>r^*$ there are no globally defined solutions because all solutions have finite time singularities in forward and backward time.  Crucially, this critical transition has nothing to do with whether the pullback attractor end-point tracks its corresponding stable QSE.

\begin{figure*}
\begin{center}
\begin{tikzpicture}
  \node (img1)  {\includegraphics[width=0.3\textwidth, trim=0.4cm 0.85cm 1.0cm 0.7cm, clip]{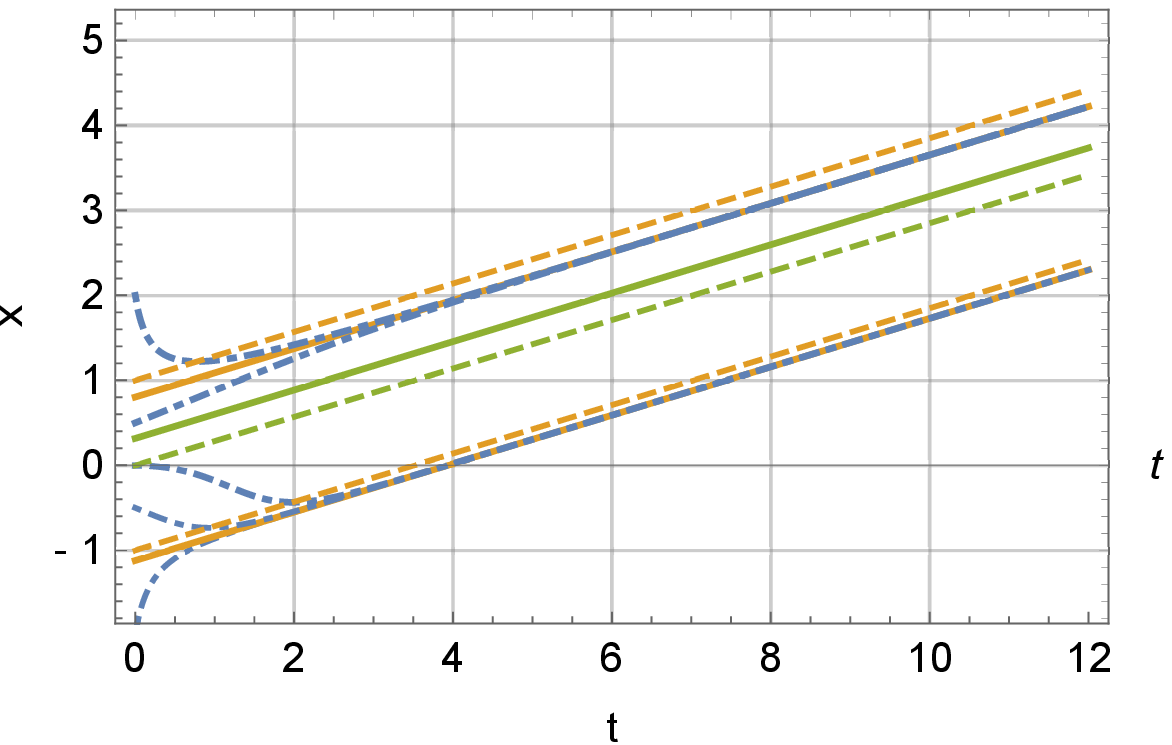}};
  \node[above=of img1, node distance=0cm, xshift=-2cm, yshift=-1.7cm] {\bf A};
  \node[below=of img1, node distance=0cm, yshift=1.1cm] {$t$};
  \node[left=of img1, node distance=0cm, rotate=90, anchor=center, yshift=-1cm] {$x$};
  \node[right=of img1,yshift=0.05cm, xshift=-.5cm] (img2)  {\includegraphics[width=0.31\textwidth, trim=0.4cm 0.85cm 1.0cm 0.7cm, clip]{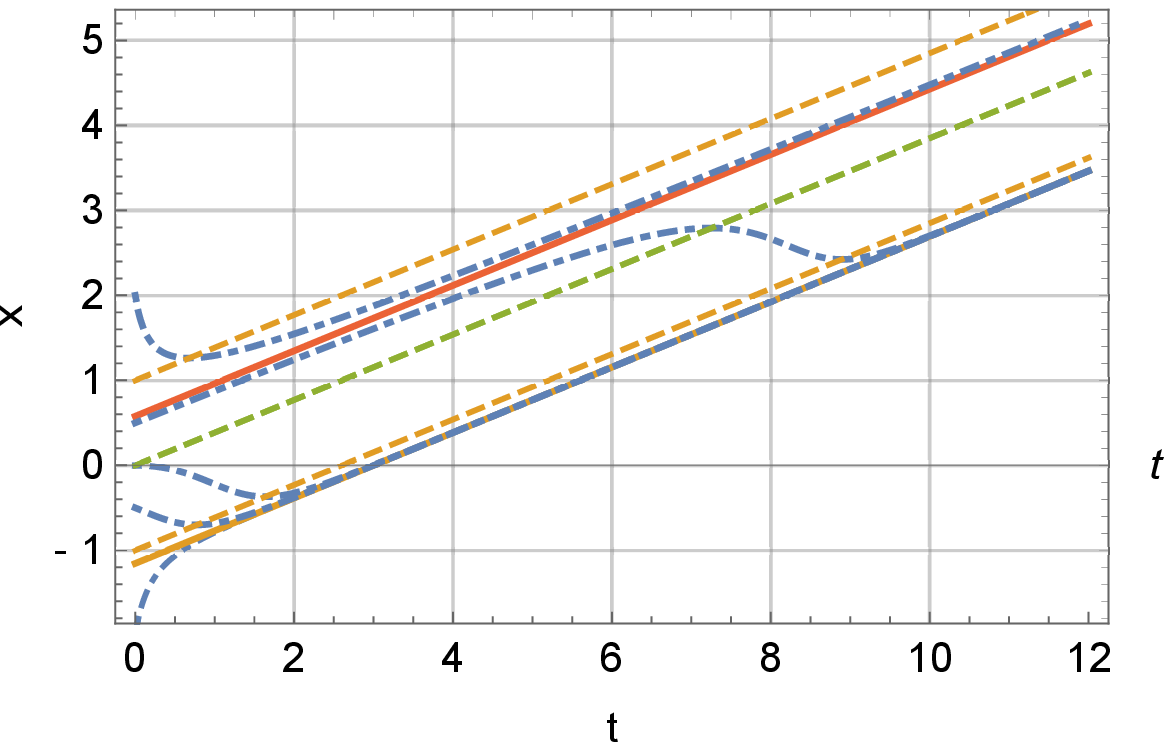}};
  \node[above=of img2, node distance=0cm, xshift=-2cm, yshift=-1.7cm] {\bf B};
  \node[below=of img2, node distance=0cm, yshift=1.1cm] {$t$};
  \node[left=of img2, node distance=0cm, rotate=90, anchor=center, yshift=-1cm] {$x$};
  \node[right=of img2,yshift=0cm,  xshift=-.5cm] (img3)  {\includegraphics[width=0.3\textwidth,  trim=0.4cm 0.85cm 1.0cm 0.7cm, clip]{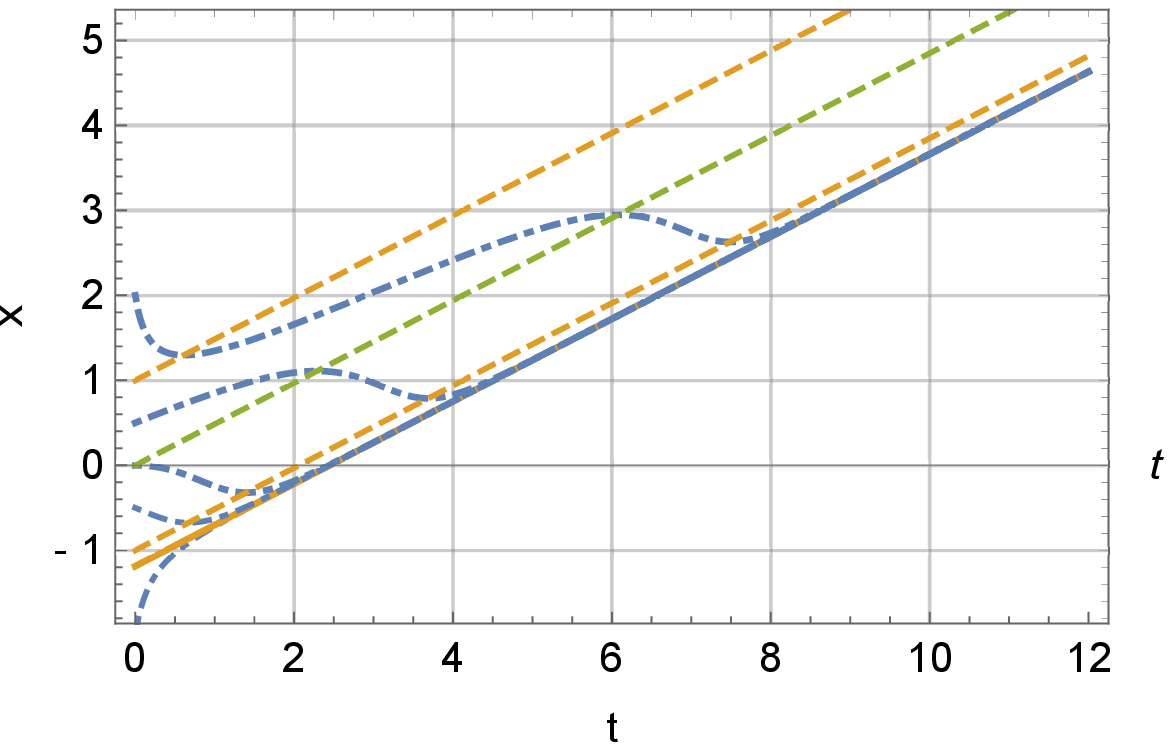}};
  \node[above=of img3, node distance=0cm, xshift=-2cm, yshift=-1.7cm] {\bf C};
  \node[below=of img3, node distance=0cm, yshift=1.1cm] {$t$};
  \node[left=of img3, node distance=0cm, rotate=90, anchor=center, yshift=-1cm] {$x$};
\end{tikzpicture}
       \caption{Time series for $\dot{x}=-(x-rt)(x-rt-1)(x-rt+1)$ showing the system's forward attracting pullback attractors (solid yellow), stable QSEs (dashed yellow), forward repelling pullback repeller (solid green), unstable QSE (dashed green), and selected solutions (blue dot-dashed). \textbf{A:} $r=-0.1+r_{+}^*$, \textbf{B:} $r=r_{+}^*=2/3\sqrt{3}$, \textbf{C:} $r=0.1+r_{+}^*2$.  As $r$ increases to the critical value $r_{+}^*$ in \textbf{B}, the top forward attracting pullback attractor and forward repelling pullback repeller collide, and are replaced by a hyperbolic globally defined solution that is neither attracting nor repelling (red).  }
    \label{fig:SN-local}
\end{center}
\end{figure*}

\subsubsection{Local Critical Transition Caused By Varying $r$}
\label{Section-SN-local}

The behavior of the bifurcation induced by varying $r$ can be local.  To see this, consider the system \eqref{EQ-system} with $N=1$ and
\begin{align}
f(x,\lambda(rt))=-(x-\lambda(rt))(x-\lambda(rt)-\mu)(x-\lambda(rt)+\mu),
\label{EQ-SN-local}
\end{align}
again with $\lambda(rt)=rt$.  Using the transformation $y(t)=x(t)-\lambda(rt)-\mu/2$, the same method as in Section \ref{Section-SN} yields a co-moving system with two stable equilibria and one unstable equilibria for 
\[|r|<\frac{2\mu^3}{3\sqrt{3}}.\]
At $r=r_{\pm}^*=\pm\frac{2\mu^3}{3\sqrt{3}}$ the system goes through a saddle node bifurcation, annihilating one of the stable equilibria and the unstable equilibrium and leaving only one stable equilibrium in the system.  In the nonauotnomous system, the equilibria correspond to forward attracting pullback attractors (stable equilibria) or forward repelling pullback repellers (unstable equilibrium).

Solutions to the nonautonomous scalar equation are plotted in Figure \ref{fig:SN-local} for $\mu=1$ and three values of $r$: $r_1=r_{+}^*-0.1$, $r_2=r_{+}^*$ and $r_3=r_{+}^*+0.1$.  As with the example from Section \ref{Section-SN}, the system goes through a nonautonomous saddle node bifurcation at the critical rate parameter value of $r=r_{+}^*$.  This bifurcation affects only the top attractor and the repeller.  The forward attraction of the bottom pullback attractor is not affected by passing through $r_{+}^*$.

\subsection{$N$-Dimensional Systems}
\label{Section-unbounded-n}

\begin{figure*}
\begin{center}
\begin{tikzpicture}
  \node (img1)  {\includegraphics[width=0.3\textwidth]{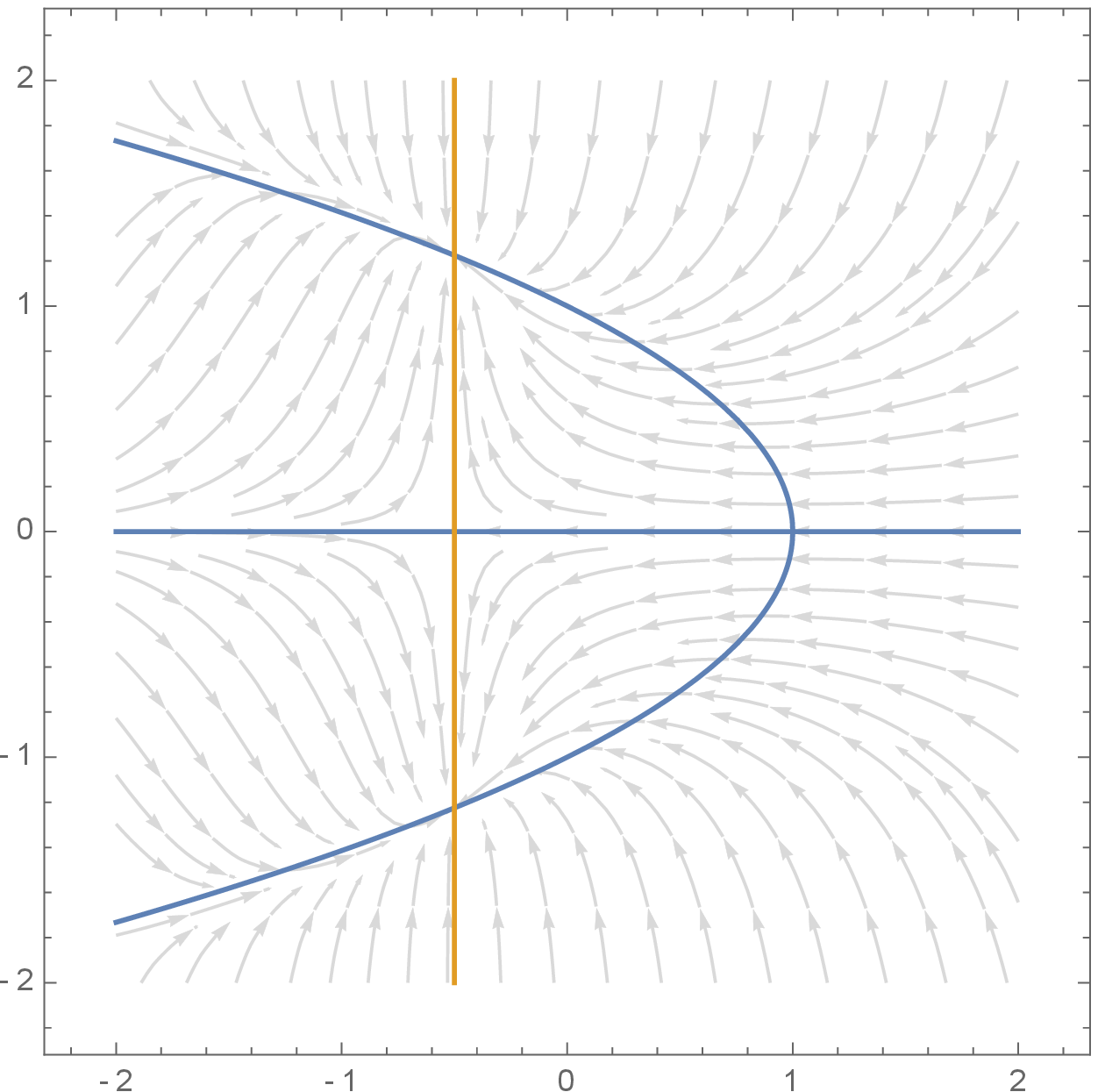}};
  \node[above=of img1, node distance=0cm, xshift=-2.1cm, yshift=-1.65cm] {\bf A};
  \node[below=of img1, node distance=0cm, yshift=1.1cm, xshift=.25cm] {$z$};
  \node[left=of img1, node distance=0cm, rotate=90, anchor=center, yshift=-1cm] {$y$};
  \node[right=of img1,yshift=0cm, xshift=-.5cm] (img2)  {\includegraphics[width=0.3\textwidth]{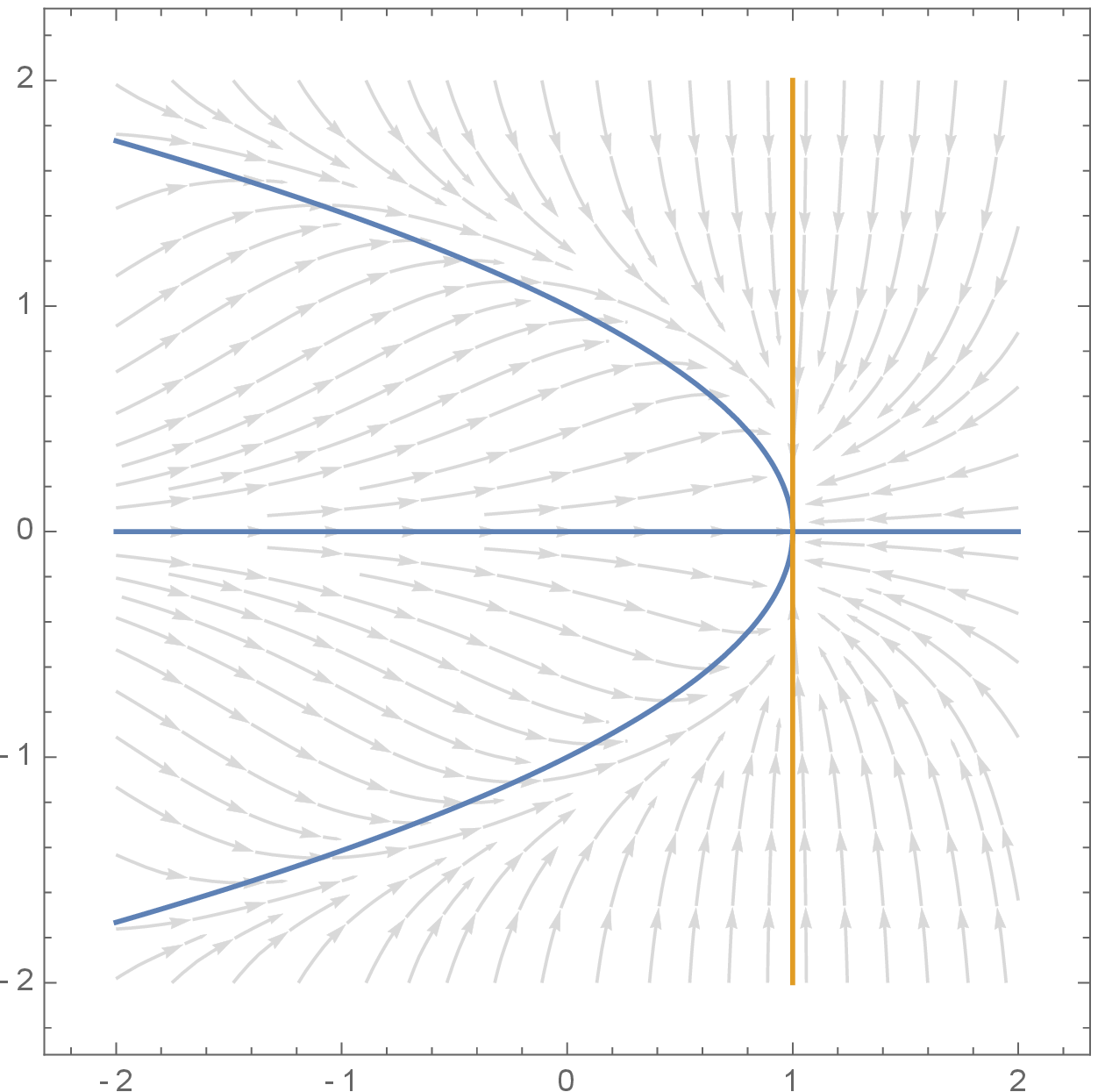}};
  \node[above=of img2, node distance=0cm, xshift=-2.1cm, yshift=-1.65cm] {\bf B};
  \node[below=of img2, node distance=0cm, yshift=1.1cm, xshift=.25cm] {$z$};
  \node[left=of img2, node distance=0cm, rotate=90, anchor=center, yshift=-1cm] {$y$};
  \node[right=of img2,yshift=0cm,  xshift=-.5cm] (img3)  {\includegraphics[width=0.3\textwidth]{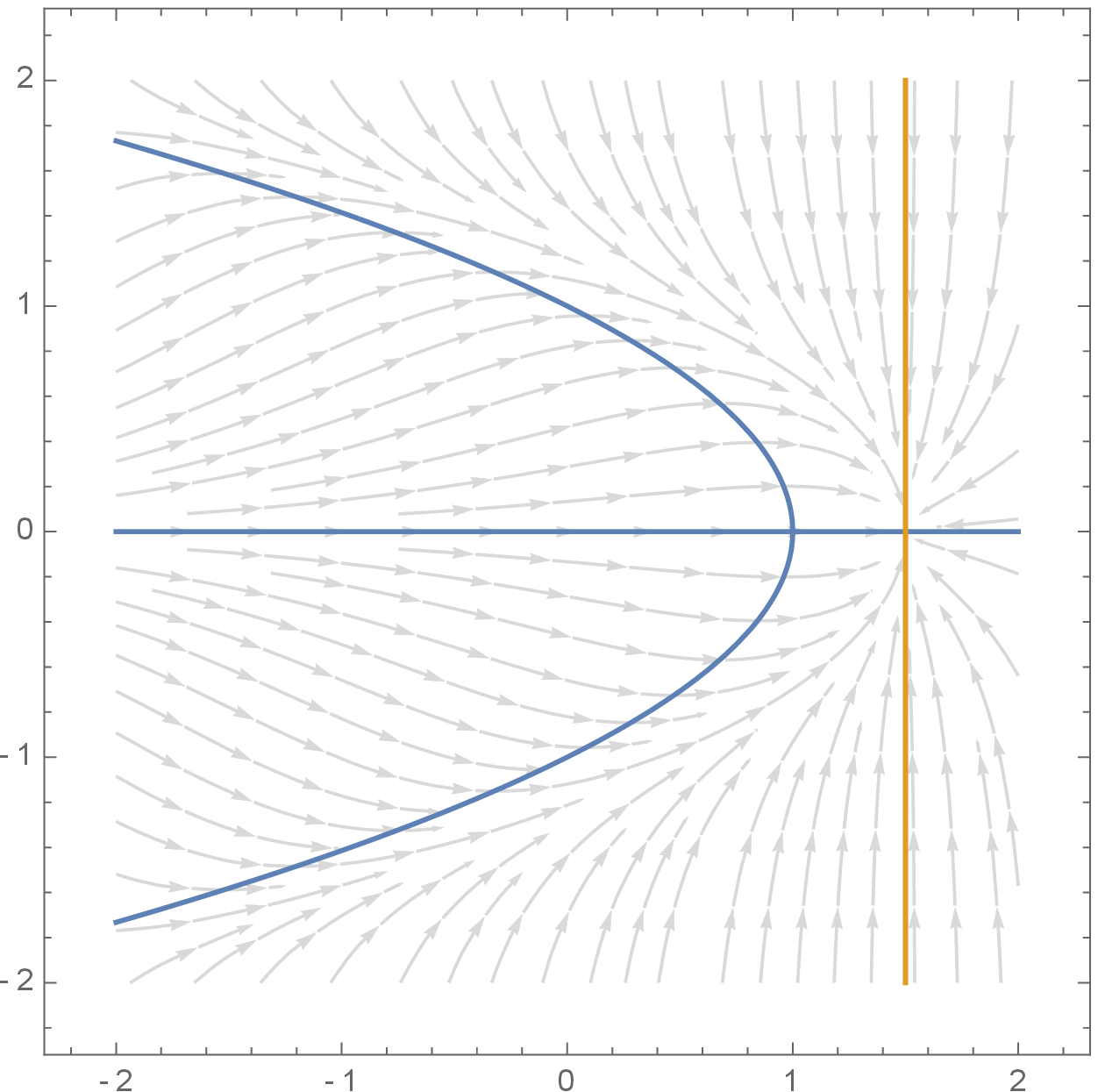}};
  \node[above=of img3, node distance=0cm,xshift=-2.1cm, yshift=-1.65cm] {\bf C};
  \node[below=of img3, node distance=0cm, yshift=1.1cm, xshift=.25cm] {$z$};
  \node[left=of img3, node distance=0cm, rotate=90, anchor=center, yshift=-1cm] {$y$};
\end{tikzpicture}
\caption{Trajectories (gray curves with arrow),  $y$-nullclines (blue solid curves), and $z$-nullclines (yellow solid curve) for the co-moving system \eqref{EQ-Ndim-auto} with $\mu=1$.  The system has a pitchfork bifurcation at $r=r^*=1$ where two sinks collide with a saddle. \textbf{A:} $r=-1/2$, \textbf{B:} $r=r^*=1$, \textbf{C:} $r=3/2$.}
    \label{fig:N-dim-auto}
    \end{center}
    \end{figure*}
    
The behavior of solutions of scalar systems not end-point tracking stable QSEs but also not tipping that we saw in the previous section are also present in $N$-dimensional systems.  This behavior is not dependent on the parameter change $\lambda$ being linear nor it is because the systems had nonautonomous saddle not bifurcations. In the next example we see that $\lambda(rt)$ may be a $p$-degree polynomial, provided the time dependence of the derivative of $\lambda$ is present in the vector field of the co-moving variable(s). Further, the system experiences rate-induced tipping due to passage through a nonautonomous pitchfork bifurcation.

Consider the two dimensional system 
\begin{equation}
\begin{aligned}
\dot x&= -\left(x + \sum_{k=1}^p{p\choose k}(rt)^k \right)-  rp\sum_{k=0}^{p-1}{p-1\choose k}(rt)^k(1-\delta(k)) \\ 
\dot y&= -y \left(x + \sum_{k=1}^p{p\choose k}(rt)^k - \mu + y^2\right)
\label{EQ-Ndim}
\end{aligned}
\end{equation}
 where $\delta(\cdot)$ is the Kronecker-delta function,
 \[\lambda(t)=\sum_{k=1}^p{p\choose k}(rt)^k\]
and $\mu>0$. As with the previous example, we may transform to the co-moving frame with the change of coordinates $z(t)=x(t)+\lambda(t)$. We have that 
 \begin{align*}\dot z&=\dot x+\dot\lambda\\
 &=\dot x+\sum_{k=1}^{p}{p\choose k}rk(rt)^{k-1}\\
  &=\dot x+\sum_{k=1}^{p}{p-1\choose k-1}rp(rt)^{k-1}.
 \end{align*}
Re-indexing the sum and plugging in $\dot x$ yields the autonomous system
\begin{equation}
\begin{aligned}
\dot z&=-z+r\\
\dot y&= -y (z - \mu + y^2).
\label{EQ-Ndim-auto}
\end{aligned}
\end{equation}
The nullclines and sample trajectories in the co-moving system are plotted for $\mu=1$ and three values of $r$ in Figure~\ref{fig:N-dim-auto}. For values of r satisfying $r<\mu$, the autonomous system has a stable equilibrium at $Z^*_{s,\pm}=(r,\pm\sqrt{\mu-r})$ and a saddle equilibrium at $Z^*_{u}=(r,0)$.  This autonomous system has a pitchfork bifurcation at $r=\mu$.

The equilibria of the co-moving system correspond to globally defined solutions for the nonautonomous system.  For $r<\mu$, the stable equilibria in the co-moving system correspond to forward attracting pullback attractors
\begin{align*}
\gamma_{r,\pm}(t)=\left[\begin{array}{c}
r-\sum_{k=1}^p{p\choose k}(rt)^k\\
\pm\sqrt{\mu - r}
\end{array}\right]
\end{align*}
in the nonautonomous system. The saddle equilibrium in the co-moving system is a hyperbolic, globally defined solution
\begin{align*}
\zeta_{r}(t)=\left[\begin{array}{c}
r-\sum_{k=1}^p{p\choose k}(rt)^k\\
0
\end{array}\right]
\end{align*}
which is neither  forward nor pullback attracting or repelling.  When $r=\mu$, all of these globally defined solutions coincide and correspond to a forward attracting pullback attractor (although the forward attraction is quite slow).  For $r>\mu$, only $\zeta_{r}(t)$, remains in the system, now as a forward attracting pullback attractor.  

As with the scalar equations, the pullback attractors do not end-point track the stable QSEs for any nonzero $r$ because the distance between the pullback attractors and the stable QSEs does not limit to zero.  Indeed, the stable QSEs are given by
\[Q_{s,\pm}(t)=\left[\begin{array}{c}
-\lambda(rt)-\dot\lambda(rt)+r  \\
\pm\sqrt{\dot\lambda(rt)-r+\mu}
\end{array}\right]\]
Then for $r<r^*$
\begin{align*}
||\gamma_{r,\pm}&(t)-Q_{s,\pm}(t)||\\
&=\left((\dot\lambda)^2+\dot\lambda+2(\mu-r)-2\sqrt{(\mu-r)(\mu-r+\dot\lambda)}\right)^{1/2}
\end{align*}
For $p=1$, the distance is given by the constant $\sqrt{r^2-r+2\mu-2\sqrt{\mu(\mu-r)}}.$ When $p>1$, the distance grows to infinity as $t\to\infty$.
Some solutions for the nonautonomous system are plotted in Figure~\ref{fig:ndim-nonauto} for $\mu=1$ and $p=1$ (Multimedia view).

\begin{figure*}
\begin{center}
  \begin{tikzpicture}
  \node (img1)  {\includegraphics[width=0.3\textwidth]{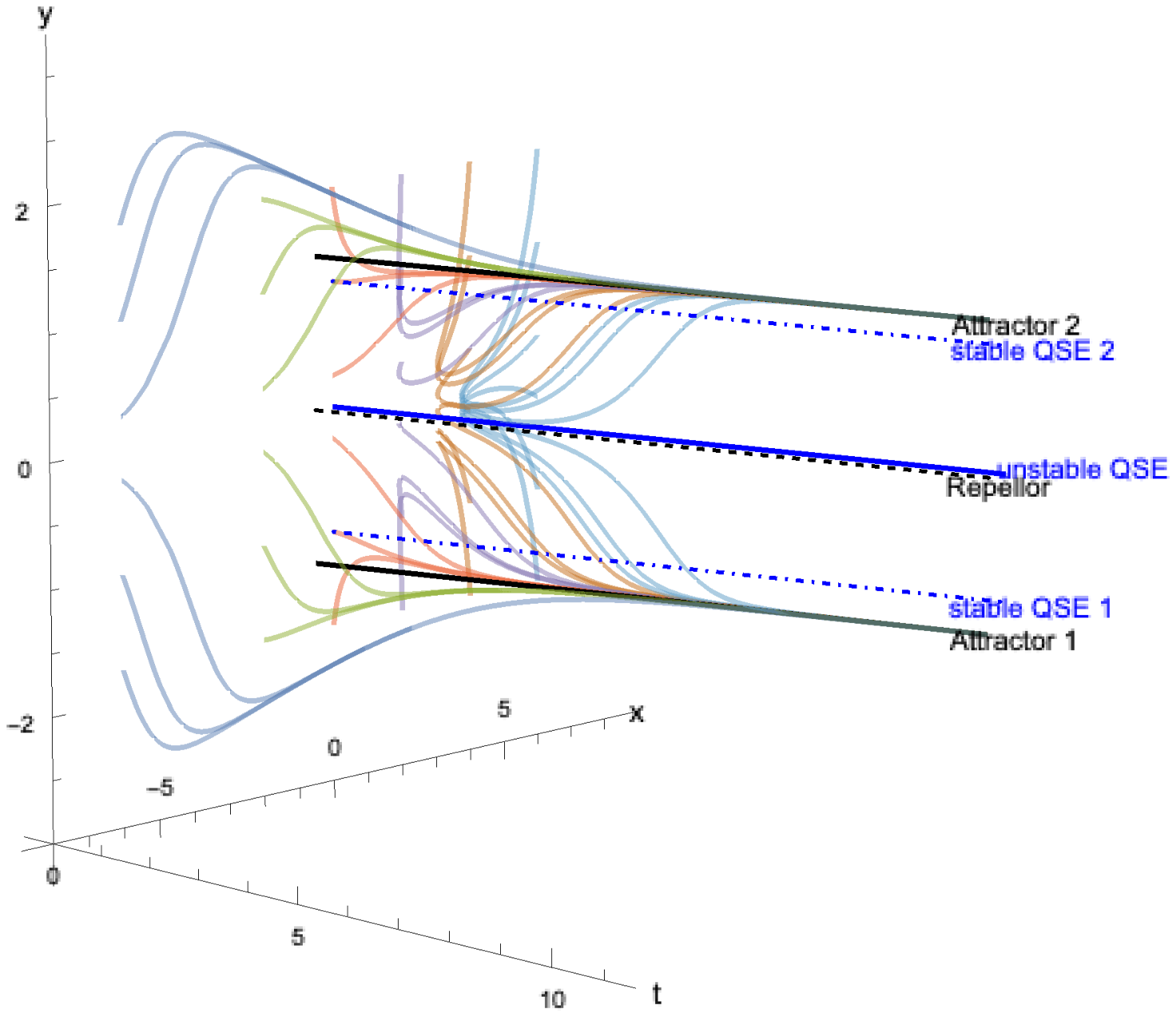}};
  \node[above=of img1, node distance=0cm, xshift=-2.1cm, yshift=-1.65cm] {\bf A};
  \node[right=of img1,yshift=0cm, xshift=-.5cm] (img2)  {\includegraphics[width=0.3\textwidth]{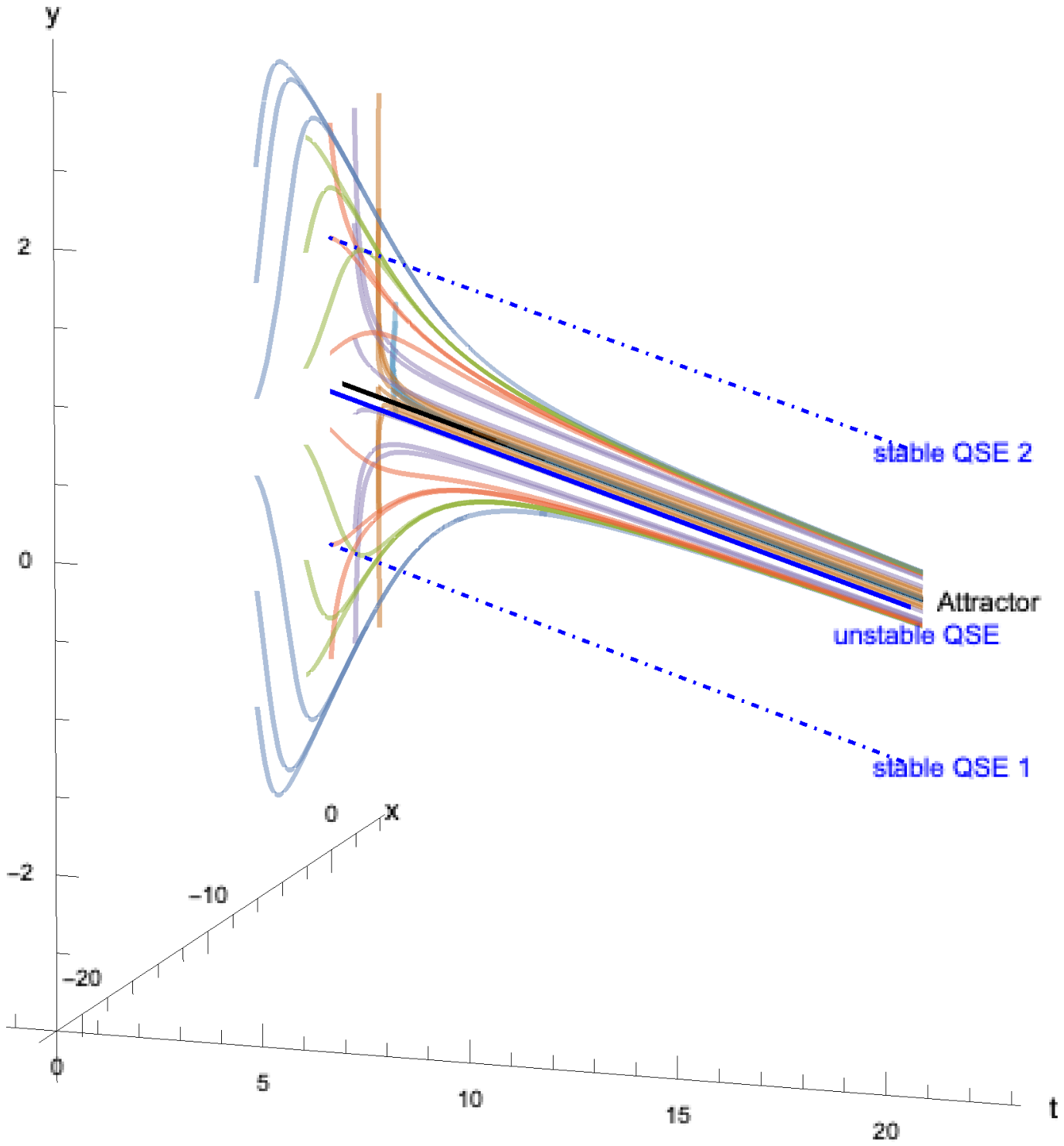}};
  \node[above=of img2, node distance=0cm, xshift=-1.9cm, yshift=-1.65cm] {\bf B};
  \node[right=of img2,yshift=0cm,  xshift=-.5cm] (img3)  {\includegraphics[width=0.3\textwidth]{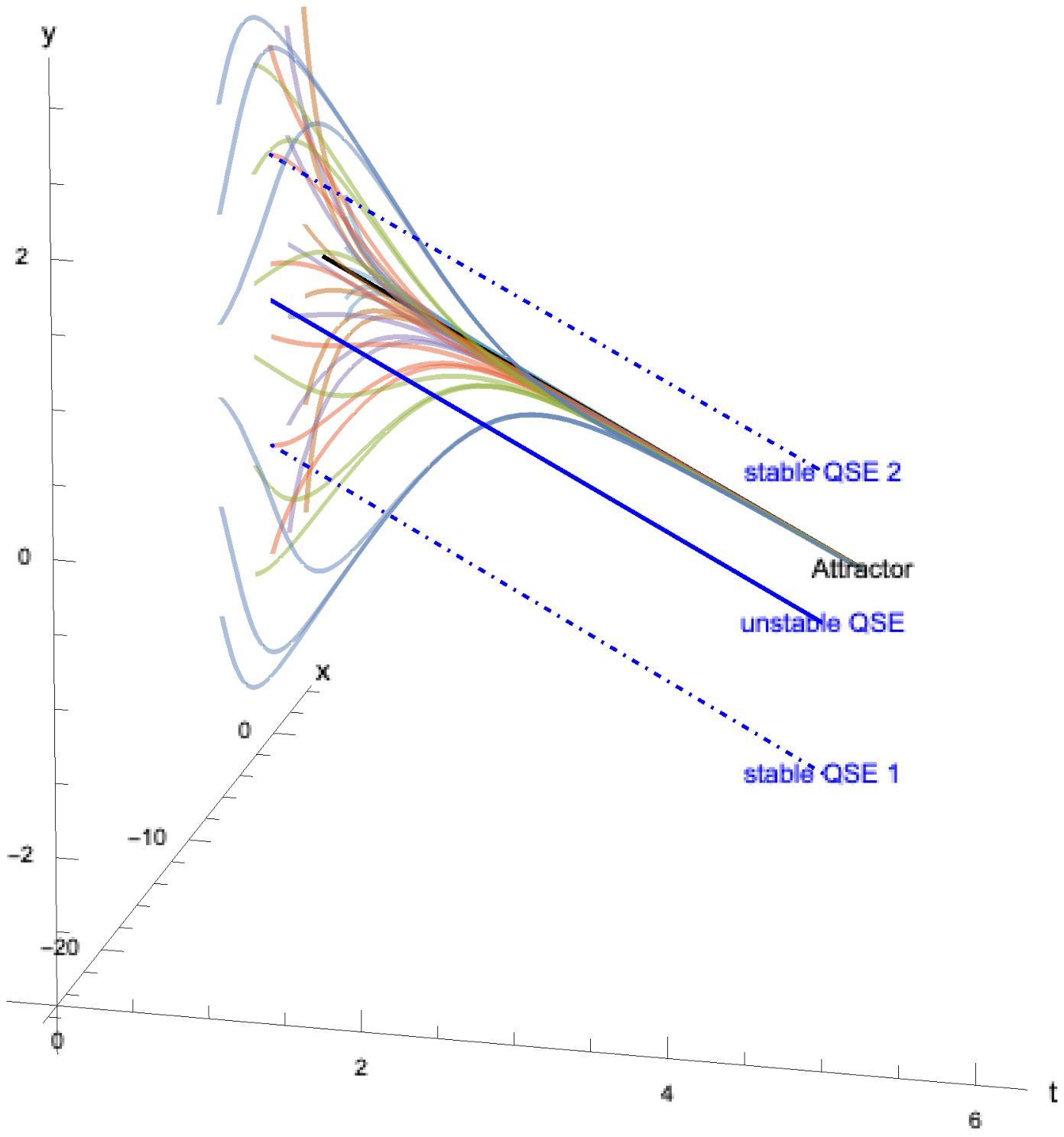}};
  \node[above=of img3, node distance=0cm,xshift=-1.9cm, yshift=-1.65cm] {\bf C};
  \end{tikzpicture}
\caption{Time Series for \eqref{EQ-Ndim} with $p=1$ and $\mu=1$ showing the system’s forward attracting pullback attractor (solid black), stable QSE(dot-dashed blue), forward repelling pullback repeller (dashed black), unstable QSE (solid blue), and selected solutions (various opaque colors). The system has a nonautonomous bifurcation at $r=r^*=1$ where a two forward attracting pullback attractors collides with the forward repelling pullback repleler in a pitchfork bifurcation. \textbf{A:} $r=-1/2$, \textbf{B:} $r=r^*=1$, \textbf{C:} $r=5$. (Multimedia view)}
    \label{fig:ndim-nonauto}
    \end{center}
    \end{figure*}

In Figure \ref{fig:ndim-nonauto} (Multimedia view) we plot time series various solutions to  \eqref{EQ-Ndim} for $p=1$ and $\mu=1$ and see a critical transition induced by varying the rate parameter $r$. We consider three values of $r$, one value below the tipping threshold ($r=-1/2$), the critical value of $r$ that causes tipping ($r^*=1$) and one value above the tipping threshold ($r=5$). We also plot forward attracting pullback attractors (solid black), forward repelling pullback repeller (dashed black), stable QSEs (dot-dashed blue), and unstable QSE (solid blue).  Tipping corresponds to a pitchfork bifurcation in the pullback attractors and pullback repeller at $r=r^*$.  For larger $r$ only one pullback attractor remains in the system.

\subsection{Tipping when a Transformation to a Co-moving System is Possible}
\label{Section-proof}

It is possible to generalize these ideas to $N$-dimensional systems with locally bounded parameter changes, provided they can be transformed through a co-moving coordinate change to an autonomous system.  
Here we consider translated systems like the ones from our examples, namely where the parameter changing in time, $\lambda(rt)$, only appears in the system with a difference to state variables.  More precisely, instead of considering the fully general case
\begin{align*}
\dot x=f(x,\lambda(rt),t),\quad x\in\mathbb R^N, \ \lambda(rt)\in\mathbb R^M
\end{align*}
we instead consider systems of the form
\begin{equation}
\begin{aligned}
\dot x=f(x-\Lambda(rt),t,\mu,r), & \quad x,\Lambda(rt)\in\mathbb R^N,  \ \mu\in\mathbb R^M, r\in\mathbb R,\\
\Lambda_i(rt)=a_i\lambda(rt), & \quad a_i,\lambda(rt)\in\mathbb R, i\in\{1,\ldots,N\}
\label{EQ-system-proof}
\end{aligned}
\end{equation}
where $\mu$ denotes the constant parameters of the system.  Although this choice is certainly restricting, similar restrictions are found in other rate-induced tipping studies and are called \emph{parameter shift systems}.\cite{Ashwin2017}  Of course we also take $f$ smooth enough to ensure existence and uniqueness of solutions. 

With these types of systems in mind, we find that if there exists a coordinate change taking the nonautonomous system \eqref{EQ-system-proof} to an autonomous one, then equilibria of the autonomous system correspond to hyperbolic, globally defined solutions in the nonautonomous system. While $\lambda(rt)=rt$ is an obvious candidate for the types of parameter change that will allow us to change to an autonomous co-moving system, this form of $\lambda$ will only result in an autonomous transformation when $f$ has no other time dependence except for that which is induced by the time dependent $\lambda$. As we saw with the example in Section \ref{Section-unbounded-n}---where $\lambda$ was a polynomial of degree $p$---more interesting $\lambda(rt)$ are possible provided the time-dependence of $f$ is cancelled out by the time dependent components of the transformation.  As such, we are not restricting to any particular form of $\lambda(rt)$. More specifically:


\begin{prop}
Suppose 
\begin{enumerate}[(i)]
\item $\lambda(rt)$ is locally bounded for all $r$
\item for $x=(x_1,x_2,\ldots,x_N)\in\mathbb R^N$ there is a set $\mathcal N\subseteq\{1,\ldots,N\}$ so that the time dependent vector $v(t)$ has components  $v_i(t)=a_i\lambda(rt)+b_i$ for some constants $a_i,  b_i\in\mathbb R$ and $i\in\mathcal N$ and $v_j(t)=0$ for $j\not\in\mathcal N$;
\item and the transformation $y=x-v(t)$ results in an autonomous system
\begin{align}
\dot y=g(y,\mu,r)=f(y,t,\mu,r)-\dot v(t),\quad y\in\mathbb R^N.
\label{EQ-comoving}
\end{align}
\end{enumerate}
Then any hyperbolic equilibria $y_{r}^*$ of the co-moving system \eqref{EQ-comoving} correspond to globally defined solutions of the original nonautonomous system \eqref{EQ-system-proof}. Furthermore,
\begin{enumerate}[(a)]
\item $y_{r}^*$ is a sink if and only if the corresponding globally defined solution of \eqref{EQ-system-proof} is a forward attracting pullback attractor; and
\item $y_{r}^*$ is a source if and only if the corresponding globally defined solution of \eqref{EQ-system-proof} is a forward repelling pullback repeller.
\end{enumerate}
\end{prop}

\begin{proof}
Let $v(t)\in\mathbb R^N$ be defined as in the statement of Proposition 2 and let $y(t)$ be a solution to \eqref{EQ-comoving}.  Then $x(t)=y(t)+v(t)$ is a solution to \eqref{EQ-system-proof} and for each hyperbolic equilibrium $y_{r}^*$  of \eqref{EQ-comoving},  $x^*(t)=y_{r}^*+v(t)$ is a globally defined solution to the original system. By a similar argument, given any solution $x(t;x_0,t_0)$ to the nonautonomous system, $y(t)=x(t;x_0,t_0)-v(t)$ is a solution to the autonomous co-moving system.

Suppose $y_{r}^*$ is a sink and take a ball $\mathcal B(y_{r}^*)$ so that all solutions with initial conditions in $\mathcal B(y_{r}^*)$ converge to $y_{r}^*$. Let $\mathbf B(t)=\mathcal B(y_{r}^*)+v(t)$ so that $\mathbf B(t)$ is the ball $\mathcal B(y_{r}^*)$ in the moving system. Fix $t\in \mathbb R$.  Given a solution $x(\tau;x_0,t_0)$ to the nonautonomous system with $t_0<t$ and initial condition $x(t_0;x_0,t_0)=x_0\in\mathbf B(t_0)$, there is a corresponding solution $y(\tau)=x(\tau;x_0,t_0)-v(\tau)$ to the autonomous co-moving system.  Furthermore, since $x(t_0;x_0,t_0)=x_0\in\mathbf B(t_0)$, we have that $y(t_0)=y_0\in\mathcal B(y_{r}^*)$. 

This motivates us to write $x_0=y_0+v(t_0)$ so that we may ``pull back'' the initial condition $x_0$ to $x_s=y_0+v(s)$ with $s<t_0$. Then the solution to the nonautonomous system through the point $x_s$ at time $s$, $x(\tau;x_s,s)$ corresponds to the solution to the co-moving autonomous system which has initial condition $y_0$ at time $s$. Due to time invariance of autonomous systems, the distance between $y(t)$ and $y_{r}^*$ is determined by the difference between the current time $t$ and the initial time $s$, thus we see that 
\begin{align*}
0=\lim_{s\to-\infty}|y(t)-y_{r}^*|=\lim_{s\to-\infty}|x(t;x_s,s)-x^*(t)|.
\end{align*}
The forward attraction of $x^*(t)$ follows from the asymptotic stability property of $y_{r}^*$.

To show the other direction we will show that if $y_{r}^*$ is not a sink, then the corresponding globally defined solution $x^*(t)$ of the nonautonoumous system is not forward attracting.

Suppose $y_{r}^*$ is not a sink.  Since we are assuming $y_{r}^*$ is a hyperbolic equilibrium, it must be that $y_{r}^*$  has an unstable manifold of dimension at least 1. 
Let $\phi(x_0,t)$ denote the flow associated with the autonomous system \eqref{EQ-comoving} and $\mathcal W^s(y_{r}^*)$ (respectively $\mathcal W^u(y_{r}^*)$) denote the stable (unstable) manifold of $y_{r}^*$. Let $\mathcal B(y_{r}^*)$ be a ball  containing $y_{r}^*$ satisfying the conditions that (1) its closure contains no other invariant sets,  (2) $\phi(\mathcal B(y_{r}^*),t)$ limits to $\mathcal W^u(y_{r}^*)$ in forward time, and (3) $\mathcal W^s(y_{r}^*)$ in backward time (or if $y_{r}^*$ is a source, then  $\phi(\mathcal B(y_{r}^*),t)$ limits to $y_{r}^*$ in backward time).  These criteria ensure that $\mathcal B(y_{r}^*)$ does not intersect with some other invariant set of the system, e.g. a periodic orbit.  Such a ball is possible to find because hyperbolic equilibria are isolated invariant sets. 

As before, let $\mathbf B(t)= \mathcal B(y_{r}^*)+v(t)$ and now also define $\mathbf W^s(t)=\mathcal W^s(y_{r}^*)+v(t)$ and $\mathbf W^u(t)=\mathcal W^u(y_{r}^*)+v(t)$. Given a solution $x(\tau;x_0,t_0)\not=x^*(t)$ to the nonautonomous system with initial condition $x(t_0;x_0,t_0)=x_0\in\mathbf B(t_0)\cap (\mathbf W^s(t_0))^c$, there is a corresponding solution $y(\tau)=x(\tau;x_0,t_0)-v(\tau)$ to the autonomous co-moving system with $y(t_0)=y_0\in\mathcal B(y_{r}^*)\cap (\mathcal W^s(y_{r}^*))^c$. Then there is some time $T$ so that  for $t>T$, $y(t)$ is no longer in the ball $\mathcal B(y_{r}^*)$ and, thus, $x(t;x_0,t_0)$ is no longer in $\mathbf B(t)$.  Since the globally defined solution $x^*(t)$ is in $\mathbf B(t)$ for all time, $x(t;x_0,t_0)$ does not limit to $x^*(t)$ in forward time.

Showing part (b) follows in a similar manner.  
\end{proof}

From the above proof, we see that  $y_{r}^*$ has an unstable manifold of dimension at least 1 if and only if the corresponding globally defined solution of \eqref{EQ-system-proof} is not forward attracting. 
Then a bifurcation in the autonomous system caused by varying the rate parameter which causes at least one of the dimensions of the stable manifold of a sink $y_{r}^*$ to become unstable is realized in the nonautonomous system as the pullback attractor losing its forward stability. This is a rate-dependent critical transition that we see in the system as solutions in the system will  converge to a different attractor or grow unbounded. Since this transition  corresponds to a bifurcation in the co-moving autonomous system, all the regular co-dimension bifurcation theory from autonomous systems for hyperbolic equilibria apply. Indeed, we see that tipping may be caused by nonautonomous saddle-node, pitchfork, transcritical, or Hopf bifurcations.  The example in Section \ref{Section-SN} is a nonautonomous saddle node bifurcation as are the examples of tipping given by Kuehn and Longo\cite{Kuehn2020} and Longo et al.\cite{Longo2021} The example in Section \ref{Section-unbounded-n} is a nonautonomous pitchfork bifurcation. Ashwin et al\cite{Ashwin2012} and later Hahn\cite{Hahn2017} analyzed rate-induced tipping in nonautonomous systems where the co-moving autonomous system has a Hopf bifurcation.

\section{Final Remarks}

Here we have proposed a new definition of rate-induced tipping which we find to be more inclusive of rate-dependent critical transitions for deterministic, continuous time systems.  We have restricted our discussion to nonautonomous systems whose corresponding autonomous systems have only hyperbolic fixed points. More work is needed to understand if the ideas presented here could be generalized to other types of systems, for example where the corresponding autonomous system has a periodic orbit or more complicated behavior such as a strange attractor.

Although the examples discussed in Section \ref{Section-unbounded} are far from exhaustive, we nonetheless find them to be illustrative of behaviors that systems with unbounded parameter change may exhibit as their parameters are varied in time. In particular, a pullback attractor not endpoint tracking a stable QSE is not indicative of a rate-dependent critical transition in systems with unbounded parameter change.  This is the main motivator for our call for a new definition of rate-induced tipping.

\begin{acknowledgments}
A portion of this work was completed during A.H-L.'s sabbatical which was partially supported by the Hutchcroft Fund and the Mathematics and Statistics Department at Mount Holyoke College.  The research of A.N.N. was supported by an NSF Mathematical Sciences Postdoctoral Research Fellowship, Award Number DMS-1902887. Both authors would like to acknowledge support by the AMS Mathematics Research Communities, Award Number DMS-1321794.
\end{acknowledgments}

\section*{Data Availability}

The data that support the findings of this study are almost all available within the article. Any data that are not available can be found from the corresponding author upon  reasonable request.

\nocite{*}
\bibliography{tipping-bib}

\end{document}